\documentclass[a4paper,12pt]{article}

\usepackage{amsthm}
\usepackage{latexsym}
\usepackage{dsfont}
\usepackage{bbm}
\usepackage{amssymb}
\usepackage{amsmath}
\usepackage{graphicx}

\numberwithin{equation}{section}

\theoremstyle{plain}
\newtheorem{theorem}{Theorem}[section]
\newtheorem{lemma}[theorem]{Lemma}
\newtheorem{corollary}[theorem]{Corollary}
\newtheorem{proposition}[theorem]{Proposition}
\theoremstyle{remark}

\theoremstyle{definition}
\newtheorem{example}[theorem]{Example}

\DeclareMathOperator{\N}{\mathbb{N}}
\DeclareMathOperator{\Z}{\mathbb{Z}}

\DeclareMathOperator{\R}{\mathbb{R}}

\DeclareMathOperator{\V}{\mathbb{V}}

\DeclareMathOperator{\Prob}{\mathbb{P}}

\DeclareMathOperator{\E}{\mathbb{E}}
\DeclareMathOperator{\G}{\mathbb{G}}
\DeclareMathOperator{\A}{\mathcal{A}}

\DeclareMathOperator{\1}{\mathbf{1}}

\DeclareMathOperator{\bT}{\mathbf{T}}
\DeclareMathOperator{\bC}{\mathbf{C}}
\DeclareMathOperator{\bX}{\mathbf{X}}

\DeclareMathOperator{\aT}{_{\mathit{a}}\!\mathit{T}}
\DeclareMathOperator{\aTi}{_{\mathit{a}}\!\mathit{T}_{\mathit{i}}}
\DeclareMathOperator{\am}{_{\mathit{a}}\!\mathit{m}}
\DeclareMathOperator{\aL}{_{\mathit{a}}\!\mathit{L}}
\DeclareMathOperator{\aS}{_{\mathit{a}}\!\mathit{S}}
\DeclareMathOperator{\aN}{_{\mathit{a}}\mathit{N}}
\DeclareMathOperator{\aW}{_{\mathit{a}}\!\mathit{W}}
\DeclareMathOperator{\Fsum}{\mathfrak{F}_{\Sigma}}

\begin{document}

\thispagestyle{empty}
\title{Fixed points of inhomogeneous smoothing transforms}
\author{	Gerold Alsmeyer\footnote{ Gerold Alsmeyer,
        Institut f\"ur Mathematische Statistik,
        Universit\"at M\"unster,
        Einsteinstra\ss e 62,
        DE-48149 M\"unster,
        Germany}		\and
        Matthias Meiners\footnote{Corresponding author: Matthias Meiners,
        Matematiska institutionen,
        Uppsala universitet,
        Box 480,
        751 06 Uppsala, Sweden.
        Email: matthias.meiners@math.uu.se.
        Research supported by DFG-grant Me 3625/1-1}}

\maketitle

\begin{abstract}
We consider the inhomogeneous version of the fixed-point equation of the smoothing transformation, that is, the equation $X \stackrel{d}{=} C + \sum_{i \geq 1} T_i X_i$, where $\stackrel{d}{=}$ means equality in distribution, $(C,T_1,T_2,\ldots)$ is a given sequence of non-negative random variables and $X_1,X_2,\ldots$ is a sequence of i.i.d.\ copies of the non-negative random variable $X$ independent of $(C,T_1,T_2,\ldots)$. In this situation, $X$ (or, more precisely, the distribution of $X$) is said to be a fixed point of the (inhomogeneous) smoothing transform. In the present paper, we give a necessary and sufficient condition for the existence of a fixed point. Further, we establish an explicit one-to-one correspondence with the solutions to the corresponding homogeneous equation with $C=0$. Using this correspondence and the known theory on the homogeneous equation, we present a full characterization of the set of fixed points under mild assumptions.
\end{abstract}
\vspace{0,1cm}


\hphantom{2010 Mathematics Subject Classification:}
Secondary:	39B22,				
				60J80				

\section{Introduction} \label{sec:Intro}

For a given sequence $(C,T) = (C,T_1,T_2,\ldots)$ of non-negative random variables we are interested in the solutions to the stochastic fixed-point equation
\begin{equation}	\label{eq:SumFP}
X \stackrel{d}{=} C + \sum_{i \geq 1} T_i X_i,
\end{equation}
where $\stackrel{d}{=}$ means equality in distribution and $X_1,X_2,\ldots$ is a sequence of i.i.d.\ copies of the non-negative random variable $X$. The sequence $X_1, X_2, \ldots$ is assumed to be independent of $(C,T)$. Any probability distribution $P$ on $[0,\infty)$ such that \eqref{eq:SumFP} holds with $X\stackrel{d}{=}P$ is called a fixed point of the smoothing transform (corresponding to $(C,T)$). This notion (in the homogeneous case) was coined by Durrett and Liggett, see \cite{DL1983}. Let $\Fsum(C,T)$ denote the set of these fixed points. In slight abuse of terminology, we will also refer to a random variable $X$ as an element of $\Fsum(C,T)$ if this is actually true for $\Prob(X \in \cdot)$.

We continue with some examples in which equations of the form \eqref{eq:SumFP} occur.

\begin{example}[Total population of a Galton-Watson process]	\label{Exa:total_population_(sub-)critical_GWP}
Let $(Z_n)_{n \geq 0}$ be a subcritical or non-trivial critical Galton-Wat\-son process with a single ancestor and total population size $X := \sum_{n \geq 0} Z_n$. Then $X$ is almost surely finite and satisfies
\begin{equation}	\label{eq:total_population_(sub-)critical_GWP}
X \stackrel{d}{=} 1 + \sum_{i \geq 1} \1_{\{Z_1 \geq i\}} X_i,
\end{equation}
where $X_{i}$ denotes the total size of the subpopulation stemming from the $i$th individual of the first generation. Since $X_{1},X_{2},\ldots$ are i.i.d.\ copies of $X$ and independent of $Z_{1}$, we see that $X$ forms indeed a solution to \eqref{eq:SumFP}
with 
\begin{equation*}
(C,T_1,T_2,\ldots) = (1,\1_{\{Z_1 \geq 1\}},\1_{\{Z_1 \geq 2\}},\ldots).
\end{equation*}
\end{example}

\begin{example}[Busy period in the M/G/1 queue]	\label{Exa:M/G/1}
Consider the M/G/1-queue with Poisson arrival process $(N(t))_{t \geq 0}$ with intensity $\lambda>0$ and i.i.d.\ service times $U_0,U_1,...$ having finite positive mean $\mu$. It is well-known that if the traffic intensity $\rho =\lambda \mu$ is less than or equal to 1, then
busy and idle periods alternate, see \textit{e.g.}\ \cite[Theorem 3 on p.\;58]{Tak1962},  \cite[Section XIV.4]{Fel1971}, or \cite[p.\;64]{AN1972}. In particular, the duration $X$ of a busy period is almost surely finite. Further, when assuming that the customer arriving at time 0 finds the server idle, $X$ satisfies the distributional equation
\begin{equation}	\label{eq:M/G/1}
X \stackrel{d}{=} U_0 + \sum_{i=1}^{N(U_0)} X_i,
\end{equation}
where $X_1, X_2, \ldots$ are i.i.d\ copies of $X$ and independent of $(N(t))_{t \geq 0}$ and $(U_n)_{n \geq 0}$, see \textit{e.g.}\ \cite[Theorem 4 on p.\;60]{Tak1962}.
\end{example}

\begin{example}[PageRank]	\label{Exa:PageRank}
In a recent publication, Volkovich and Litvak \cite{VL2010}, in their analysis of Google's \texttt{PageRank} algorithm, are led to an equation of type \eqref{eq:SumFP} for the Page\-Rank $X$, a certain importance measure for a randomly chosen web site. Roughly speaking, viewing the World Wide Web as an oriented graph, where nodes are pages and edges are links, the PageRank is defined as the stationary distribution of an ``easily bored'' surfer who either moves at random (with some probability $c$) from a current site to a neighboring one along an outging edge, or makes a \emph{teleportation jump} (with probability $1-c$), which means that he starts afresh by picking any node of the graph in accordance with some distribution, called \emph{teleportation distribution}. However, unlike our work, the authors focus on the tail behaviour of $X$. For further details as well as further references, the reader is referred to \cite{PB1998} and also to \cite{VL2010}.
\end{example}

If $C=0$ we obtain the homogeneous version of Eq.\ \eqref{eq:SumFP}, viz.
\begin{equation}	\label{eq:SumFP_hom}
X \stackrel{d}{=} \sum_{i \geq 1} T_i X_i.
\end{equation}
Both equations \eqref{eq:SumFP} and \eqref{eq:SumFP_hom} lead to functional equations when stated in terms of Laplace transforms. Indeed, \eqref{eq:SumFP} holds iff the Laplace transform $\psi$ of $X$ satisfies
\begin{equation}	\label{eq:SumFE}
\psi(t) = \E e^{-tC} \, \prod_{i \geq 1} \psi(T_i t)	\quad	(t \geq 0),
\end{equation}
while \eqref{eq:SumFP_hom} is equivalent to
\begin{equation}	\label{eq:SumFE_hom}
\psi(t) = \E \prod_{i \geq 1} \psi(T_i t)	\quad	(t \geq 0).
\end{equation}
In slight abuse of terminology, we henceforth write $\psi \in \Fsum(C,T)$ for a Laplace transform $\psi$ when we mean that the distribution pertaining to $\psi$ is an element of $\Fsum(C,T)$.

Eq.\ \eqref{eq:SumFE_hom} has been studied in great detail in the literature, the most recent reference being \cite{ABM2010}, where a characterization of the set of monotone solutions to the functional equation is presented with only very weak assumptions on the sequence $T_1,T_2,\ldots$ being in force. For earlier contributions see \cite{Big1977,DL1983,BK1997,Liu1998} to name but a few. A more comprehensive overview of the existing literature on the homogeneous equation can be found in \cite{ABM2010}. Work on the inhomogeneous equation started only quite recently. \cite{JO2010} and \cite{VL2010} study \eqref{eq:SumFP} under more restrictive assumptions on the sequence $(C, T_1, T_2, \ldots)$, whereas \cite{JO2010b} deals with \eqref{eq:SumFP} in essentially the same generality as here. All three references focus on the tail behavior of $X$ under varying conditions on $(C,T_1,T_2,\ldots)$, while the goal of the present paper is rather to determine the set of solutions to \eqref{eq:SumFP}. Using a similar approach, this has also been done by Spitzmann in his PhD thesis \cite{Spi2010}, which includes a chapter on determining the solutions to \eqref{eq:SumFP} via a look at the associated functional equation \eqref{eq:SumFE}. Moreover, Spitzmann studies the functional equation in the more general case where $\psi$ is a non-negative decreasing function and the inhomogeneity factor $e^{-Ct}$ is replaced by one of the more general form $g(tC)$ for some decreasing function $g$. This allows him to cover the related $\min$-type equation
\begin{equation}	\label{eq:MaxFP}
X	\stackrel{d}{=}	\inf\{X_i/T_i: i \geq 1, T_i > 0\} \wedge C
\end{equation}
and the related $\max$-type equation as well. On the other hand, this greater generality is at the cost of clarity which is why we decided to confine ourselves to the sum-type equation. Most of the arguments given here could be modified, however, to cover the more general functional equation as well.

The paper is organised as follows. The next section is devoted to an introduction of a weighted branching model which allows the explicit iteration of equations \eqref{eq:SumFP} and \eqref{eq:SumFP_hom} on a distinguished probability space. In Section \ref{sec:Fsum}, we define a ``minimal solution'' $W^*$ (see Eq.\ \eqref{eq:W*}), which is a function of the weighted branching process. It turns out that $\Fsum(C,T) \not= \emptyset$ if and only if $W^*$ is almost surely finite in which case $W^*$ forms indeed a solution (Theorem \ref{Thm:Characterization_of_Fsum_neq_empty}). Using this result as a starting point, two questions must be addressed. The first one is to investigate when $W^*$ is almost surely finite. The second one is to describe the set of solutions in the case when $W^*$ is almost surely finite. In Section \ref{sec:Disintegration}, we provide a technique, called \emph{disintegration}, that is useful in this endeavour. Disintegration particularly leads to an explicit one-to-one correspondence between the solutions of \eqref{eq:SumFP} and its homogeneous counterpart \eqref{eq:SumFP_hom} (Theorem \ref{Thm:M=exp(-W*)xM_{hom}}). Section \ref{sec:known_and_simple} treats simple and known cases that are excluded from the subsequent analysis. Sections \ref{sec:nec_conditions} and \ref{sec:suf_conditions} provide necessary and sufficient conditions, respectively, for $W^*<\infty$ almost surely. The main results are Theorem \ref{Thm:Liu_better} and Theorem \ref{Thm:suf_conditions}. Section \ref{sec:description} provides a description of $\Fsum(C,T)$ in the remaining cases under mild conditions (Theorem \ref{Thm:characterization_of_Fsum}). In Section \ref{sec:Applications} we will discuss some applications and particularly return to the examples given above. The final section contains some concluding remarks.

\section{Iterating the fixed-point equation: weighted branching}	\label{sec:WBM}

To deploy iteration in the study of a functional equation is a natural tool which, in the case of Eqs \eqref{eq:SumFP} and \eqref{eq:SumFP_hom}, leads to a weighted branching model.

Let $\V := \bigcup_{n \in \N_0} \N^n$ be the infinite Ulam-Harris tree, where $\N := \{1,2,\ldots\}$, $\N_0 = \N \cup\{0\}$, and $\N^0 = \{\varnothing\}$ is the set containing the empty tuple only. Abbreviate $v = (v_1,\ldots,v_n)$ by $v_1 \ldots v_n$ and write $v|k$ for the restriction of $v$ to the first $k$ entries, that is, $v|k := v_1 \ldots v_k$, $k \leq n$. If $k > n$, put $v|k := v$. Write $vw$ for the vertex $v_1 \ldots v_n w_1 \ldots w_m$ where $w = w_1 \ldots w_m$. In this situation, we say that $v$ is an ancestor of $vw$. The length of a node $v$ is denoted by $|v|$, thus $|v|=n$ iff $v\in\N^{n}$. Next, let $\bC\otimes\bT := ((C(v),T(v)))_{v \in \V}$ be a family of i.i.d.\ copies of $(C,T)$, where $(C(\varnothing),T(\varnothing))=(C,T)$. We refer to $(C,T)=(C,T_1,T_2,\ldots)$ as the \emph{basic sequence (of the weighted branching model)} and interpret $C(v)$ as a weight attached to the vertex $v$ and $T_i(v)$ as a weight attached to the edge $(v,vi)$ in the infinite tree $\V$. Then define $L(\varnothing) := 1$ and, recursively,
\begin{equation}	\label{eq:L(v)}
L(vi):= L(v) T_i(v)
\end{equation}
for $v \in \V$ and $i \in \N$. We interpret $L(v)$ as the total multiplicative weight of the unique path from the root $\varnothing$ to $v$. We can transform the setting into an additive rather than a multiplicative by defining $S(v) := -\log L(v)$, $v \in \V$ (where $-\log 0 := \infty$). The points $S(v)$ with $S(v) < \infty$ define a classical branching random walk, see \textit{e.g.}\ \cite{Big1977,Big1998}. For $n\in\N_0$, let $\A_n$ denote the $\sigma$-algebra generated by the $(C(v),T(v))$, $|v|<n$. Put also $\A_{\infty}:=\sigma(\A_n:n \geq 0)$ $=\sigma(\bC\otimes\bT)$.

Further, we assume the existence of a family $\bX = (X(v))_{v \in \V}$ of i.i.d.\ copies of $X$, independent of $\bC \otimes \bT$. Then, by construction, $n$fold iteration of \eqref{eq:SumFP} can be expressed in terms of the weighted branching model:
\begin{equation}	\label{eq:SumFP_iterated}
X \stackrel{d}{=} \sum_{|u|<n} L(u)C(u) + \sum_{|v|=n} L(v) X(v).
\end{equation}
In the homogeneous case, the first sum on the right-hand side vanishes and \eqref{eq:SumFP_iterated} becomes
\begin{equation}	\label{eq:SumFP_hom_iterated}
X	\stackrel{d}{=}	\sum_{|v|=n} L(v) X(v).
\end{equation}

Let us finally introduce the shift operators $[\cdot]_u$, $u \in \V$. Given any function $\Psi=\psi(\bC \otimes \bT)$ of the weight family $\bC \otimes \bT$ pertaining to $\V$, define
\begin{equation*}
[\Psi]_u := \psi((T(uv))_{v \in \V})
\end{equation*}
to be the very same function but for the weight ensemble pertaining to the subtree rooted at $u \in \V$.
Any branch weight $L(v)$ can be viewed as such a function, and we thus have
$[L(v)]_u = T_{v_1}(u) \cdot ... \cdot T_{v_n}(uv_1 ... v_{n-1})$ if $v = v_1 ... v_n$.

\section{A Characterization of $\Fsum(C,T) \not = \emptyset$}	\label{sec:Fsum}

Eq.\ \eqref{eq:SumFP_iterated} suggests that
\begin{equation}	\label{eq:W*}
W^* := \sum_{v \in \V} L(v) C(v)
\end{equation}
is a solution to \eqref{eq:SumFP}. Indeed,
\begin{align}	\label{eq:W*_solves_SumFP}
W^* &= \sum_{v \in \V} L(v) C(v) \notag\\
&= C + \sum_{i \geq 1} T_i \sum_{v \in \V} [L(v)]_i C(iv)	\notag	\\
&= C + \sum_{i \geq 1} T_i [W^*]_i\quad	\text{almost surely,}
\end{align}
where the $[W^*]_i$ are i.i.d.\ copies of $W^*$ and independent of $(C,T)$. Moreover, by non-negativity, $W^*$ is well-defined but may be infinite with positive probability. Therefore, $W^*$ is an element of $\Fsum(C,T)$ if $W^* < \infty$ almost surely, that is, almost sure finiteness of $W^*$ is sufficient for $\Fsum(C,T) \not = \emptyset$. In turn, if $\Fsum(C,T) \not = \emptyset$, one can pick $P \in \Fsum(C,T)$ and assume that $X(v)\stackrel{d}{=}P$ for all $v \in \V$. Then \eqref{eq:SumFP_iterated} implies
\begin{equation*}
X \stackrel{d}{\geq} \sum_{|u|<n} L(u)C(u)
\end{equation*}
for all $n \geq 0$ and, therefore, $X \stackrel{d}{\geq} W^*$. Here $\stackrel{d}{\ge}$ is to be understood in the ordinary sense of stochastic domination. 
Since $X$ is almost surely finite, the same follows for $W^{*}$. In other words, almost sure finiteness of $W^*$ is also necessary for $\Fsum(C,T) \not = \emptyset$.  We have thus proven the following theorem:

\begin{theorem}	\label{Thm:Characterization_of_Fsum_neq_empty}
$\Fsum(C,T) \not = \emptyset$ is equivalent to $W^* < \infty$ almost surely.
\end{theorem}

Two natural questions arise from this theorem:
\begin{itemize}
	\item[(1)]
		Study when $\Prob(W^*<\infty)=1$ holds.
	\item[(2)]
		Provide a description of $\Fsum(C,T)$ if nonempty.
\end{itemize}
Question (1) will be investigated in Sections \ref{sec:nec_conditions} and \ref{sec:suf_conditions}, followed by a full answer to Question (2) under weak assumptions on $(C,T)$ in Section \ref{sec:description}.

\section{Disintegration}	\label{sec:Disintegration}

Our next task is to establish a one-to-one correspondence between the sets $\Fsum(C,T)$ and $\Fsum(0,T)$. This will be accomplished via a disintegration technique and already provides a partial answer to Question (2).

The functional equation \eqref{eq:SumFE} for the Laplace transform $\psi$ of $X$ after $n$ iterations becomes
\begin{equation}	\label{eq:SumFE_iterated}
\psi(t) = \E \Bigg[ \exp\Bigg(-t \sum_{|u|<n} L(u) C(u)\Bigg) \prod_{|v|=n} \psi(L(v) t) \Bigg]
\qquad	(t \in \R).
\end{equation}
For $\psi \in \Fsum(C,T)$, we define the associated \emph{multiplicative martingale} by
\begin{equation} \label{eq:disintegrated}
M_n(t) :=	M_n(t,\bC \otimes \bT)
:= \exp \Bigg(-t \sum_{|u|<n}L(u) C(u)\Bigg) \cdot \prod_{|v|=n} \psi(L(v) t),
\quad n \geq 0.
\end{equation}
The next lemma asserts that $(M_n(t))_{n \geq 0}$ does indeed constitute a martingale:

\begin{lemma}	\label{Lem:Disintegration}
Let $\psi \in \Fsum(C,T)$ and $t \in [0,\infty)$. Then $(M_n(t))_{n \geq 0}$ forms a $[0,1]$-valued martingale with respect to $(\A_n)_{n \geq 0}$ and thus converges almost surely and in mean to a random variable $M(t) = M(t,\bC \otimes \bT)$ satisfying
\begin{equation}	\label{eq:disintegration_integrated}
\E M(t) = \psi(t).
\end{equation}
Moreover, $\liminf_{n \to \infty} M_n(\cdot)$ is a $\bC \otimes \bT$-measurable process such that almost every path is the Laplace transform of a probability distribution on $[0,\infty)$.
\end{lemma}
\begin{proof}
A straightforward generalization of the calculations in the proof of Theorem 3.1 in \cite{BK1997} yields the martingale property for $(M_n(t))_{n \geq 0}$ and thus, by standard theory, the convergence assertions as well as \eqref{eq:disintegration_integrated}.

For the sake of definiteness everywhere, let $M(t,\mathbf{c}\otimes\mathbf{t}) := \liminf_{n \to \infty} M_n(t,\mathbf{c}\otimes\mathbf{t})$ for $\mathbf{c} \otimes \mathbf{t} \in (\R_{\geq 0}^{\N_0})^{\V}$. Now choose a Borel measurable set $A \subset (\R_{\geq 0}^{\N_0})^{\V}$ satisfying $\Prob(\bC \otimes \bT \in A) = 1$ and such that $M_n(t, \mathbf{c} \otimes \mathbf{t}) \to M(t, \mathbf{c} \otimes \mathbf{t})$ for all rational $t \geq 0$ and $\mathbf{c} \otimes \mathbf{t} \in A$. Fix $\mathbf{c} \otimes \mathbf{t} \in A$. $(M_n(\cdot, \mathbf{c} \otimes \mathbf{t}))_{n \geq 0}$ is a sequence of Laplace transforms of probability distributions on $[0,\infty)$. Pick any vaguely convergent subsequence of this sequence of distributions and let $\Psi(\cdot,\mathbf{c} \otimes \mathbf{t})$ denote the Laplace transform of its limit. Then, necessarily, $\Psi(t,\mathbf{c} \otimes \mathbf{t}) = M(t, \mathbf{c} \otimes \mathbf{t})$ for all rational $t > 0$. Moreover, as $\Psi(\cdot, \mathbf{c} \otimes \mathbf{t})$ is continuous, it is uniquely determined by its values at rational arguments and, thus, by $M(\cdot, \mathbf{c} \otimes \mathbf{t})$. Hence, any vaguely convergent subsequence of the sequence of distributions associated with $(M_n(\cdot, \mathbf{c} \otimes \mathbf{t}))_{n \geq 0}$ has the same limit. Therefore, by a combination of the Helly-Bray theorem and the continuity theorem for Laplace transforms, $M_n(t, \mathbf{c} \otimes \mathbf{t}) \to \Psi(t, \mathbf{c} \otimes \mathbf{t})$ for all $t > 0$. Thus $\Psi := \Psi(\cdot,\bC \otimes \bT)$ equals $M$ almost surely and almost every path of $\Psi$ is the Laplace transform of a (possibly degenerate) probability measure on $[0,\infty)$. It remains to prove that $\Psi(0) = 1$ almost surely. To this end, notice that
\begin{equation*}
\| 1-\Psi(t)\|_1	=	1-\E\Psi(t)	=	1-\E M(t)	=	1-\psi(t)	\to 0	\quad	\text{as }	t \to 0,
\end{equation*}
that is, $\Psi(t) \to 1$ in $\mathcal{L}^1$. In particular, there exists a sequence $t_n \downarrow 0$ such that $\Psi(t_n) \to 1$ almost surely as $n \to \infty$. By the monotonicity of $\Psi$, this implies that $\Psi(t) \to 1$ almost surely as $t \downarrow 0$ and, thus, $\Psi(0) = 1$ almost surely. 
\end{proof}

We call the stochastic process $M=(M(t))_{t \geq 0}$ the \emph{disintegration of $\psi$} and also a \emph{disintegrated fixed point}. According to Lemma \ref{Lem:Disintegration}, the finding of all disintegrated fixed points also provides us with a description of $\Fsum(C,T)$. Moreover, disintegrations have the useful property of satisfying a pathwise version of the functional equation \eqref{eq:SumFE}.

\begin{lemma} \label{Lem:disintSFPE}
Let $\psi \in \Fsum(C,T)$ and let $M$ denote its disintegration. Then
\begin{equation} \label{eq:disintegrated_FPE}
M(t) =	\exp\Bigg(-t \sum_{|u|<n} L(u)C(u)\Bigg) \cdot \prod_{|v|=n} [M]_v(L(v)t)
\quad	\text{almost surely}
\end{equation}
for each $t \in \R$ and $n \in \N_0$.
\end{lemma}
\begin{proof}
Only minor changes are needed to derive the result from the proof of Lemma 5.2 in \cite{AM2009}.
\end{proof}

Lemma \ref{Lem:disintSFPE} has the following corollary:

\begin{corollary}	\label{Cor:bounded_above}
For any $\psi \in \Fsum(C,T)$ with disintegration $M$ and any $t \geq 0$, we have
\begin{equation}
M(t)	\leq	\exp(-tW^*)	\quad	\text{almost surely.}
\end{equation}
\end{corollary}
\begin{proof}
For any $t \ge 0$, by \eqref{eq:disintegrated_FPE},
\begin{equation*}
M(t)
\leq	\exp\Bigg(-t \sum_{|u|<n} L(u)C(u)\Bigg)
\underset{n \to \infty}{\longrightarrow}		\exp(-tW^*)
\quad	\text{almost surely.}
\end{equation*}
\end{proof}

We now use this corollary to prove a one-to-one correspondence between $\Fsum(C,T)$ and the set of disintegrations of the elements of $\Fsum(0,T)$. As the latter set in turn is one-to-one to the set $\Fsum(0,T)$ itself (by Lemma \ref{Lem:Disintegration}), the next result also proves the announced one-to-one correspondence between $\Fsum(0,T)$ and $\Fsum(C,T)$.

\begin{theorem}	\label{Thm:M=exp(-W*)xM_{hom}}
If $\Fsum(C,T) \not = \emptyset$, then any disintegration $M$ of a solution $\psi$ to \eqref{eq:SumFP} has a representation of the form
\begin{equation}	\label{eq:M=exp(-W*)xM_{hom}}
M(t)	=	\exp(-tW^*) M_{\mathrm{hom}}(t)	\quad	\text{almost surely}	\quad	(t \geq 0)
\end{equation}
where $M_{\hom}$ denotes the disintegration of an element of $\Fsum(0,T)$.
Conversely, any Laplace transform $\psi$ obtained by taking the expectation in \eqref{eq:M=exp(-W*)xM_{hom}} defines a solution to \eqref{eq:SumFP}.
\end{theorem}
\begin{proof}
We first prove the converse part of the theorem. Let $M_{\mathrm{hom}}$ denote the disintegration of an arbitrary homogeneous solution $\varphi$. Then, with $M(t)$ as defined in \eqref{eq:M=exp(-W*)xM_{hom}}, we infer
\begin{eqnarray*}
M(t)
& = &
\exp(-tW^*) M_{\mathrm{hom}}(t)	\\
& = &
\exp\Bigg(-t\bigg(C+\sum_{i \geq 1} T_i [W^*]_i\bigg)\Bigg) \prod_{i \geq 1}[M_{\mathrm{hom}}]_i(T_i t)	\\
& = &
\exp(-t C) \prod_{i \geq 1} \exp(-tT_i[W^*]_i) [M_{\mathrm{hom}}]_i(T_i t),
\end{eqnarray*}
where we have used that $M$ solves a disintegrated version of the homogeneous functional equation, see Lemma \ref{Lem:disintSFPE} with $C=0$. Taking first the conditional expectation given $\A_1$ and then the unconditional expectation yields that $\psi(t) := \E M(t)$ satisfies \eqref{eq:SumFE}. Being a mixture of Laplace transforms of probability distributions on $[0,\infty)$ by Lemma \ref{Lem:Disintegration}, $\psi$ is the Laplace transform of a proper distribution on $[0,\infty)$ and thus an element of $\Fsum(C,T)$.

Now suppose that $\psi \in \Fsum(C,T)$ and denote by $(M_n(t))_{n \geq 0}$ the corresponding multiplicative martingale with almost sure limit $M$. Then
\begin{equation*}
M_n(t) =\exp\Bigg(-t \sum_{|u|<n} L(u)C(u) \Bigg) \, \prod_{|v|=n} \psi(L(v)t)
\leq	\prod_{|v|=n} \psi(L(v)t) =: \Phi_n(t).
\end{equation*}
Here $\exp(-t \sum_{|u|<n} L(u)C(u))$ converges to $\exp(-tW^*)$, which is positive by Theorem \ref{Thm:Characterization_of_Fsum_neq_empty}. On the other hand, $M_n(t)$ tends to $M(t)$ as $t \to \infty$. Consequently, $\Phi_n(t)$ tends to $\Phi(t):=M(t)/\exp(-tW^*)$ and $0 \leq \Phi(t) \leq 1$. From Lemma \ref{Lem:disintSFPE} we infer that $M$ satisfies \eqref{eq:disintegrated_FPE}. Using this and the definition of $W^*$, we get
\begin{align*}
\Phi(t)
&=
\frac{M(t)}{\exp(-tW^*)}\\
&= \frac{\exp\Big(-t \sum_{|u|<n} L(u)C(u)\Big) \cdot \prod_{|v|=n} [M]_v(L(v)t)}
{\exp\left(-t \sum_{|u|<n} L(u) C(u) - t\sum_{|v|=n} L(v)[W^*]_v\right)}	\\
&= 
\frac{\prod_{|v|=n} [M]_v(L(v)t)}{\prod_{|v|=n} \exp(-tL(v)[W^*]_v)} \\
&= \prod_{|v|=n}	[\Phi]_v(L(v)t)	\quad	\text{almost surely.}
\end{align*}
Taking expectations yields that $\varphi(t) = \E \Phi(t)$ is a solution to \eqref{eq:SumFE_hom}. Moreover, $\varphi$ is the Laplace transform of a probability distribution on $[0,\infty)$ as a mixture of Laplace transforms of probability distributions on $[0,\infty)$. Denote by $M_{\mathrm{hom}}$ the disintegration of $\varphi$. It remains to prove that $M_{\mathrm{hom}}$ is a version of the process $\Phi$. But this is immediate from the following calculation based on an application of the martingale convergence theorem:
\begin{align*}
M_{\mathrm{hom}}(t)
&=
\lim_{n \to \infty} \prod_{|v|=n} \varphi(L(v)t)\\
&= \lim_{n \to \infty} \E\Big[ \prod_{|v|=n} [\Phi]_v(L(v)t)\Big|\, \A_n\Big]	\\
&= 
\lim_{n \to \infty} \E[ \Phi(t) | \A_n]
= \Phi(t)	\quad	\text{almost surely},
\end{align*}
where we have utilized that $\Phi(t)$ is $\A_{\infty}$-measurable.
\end{proof}

\section{Known and simple cases}	\label{sec:known_and_simple}

Not surprisingly, there is no global answer to Questions (1) and (2). Rather, we will have to distinguish between different situations, in which quite different qualitative behaviours ensue.

The first important special case of Eq.\ \eqref{eq:SumFP}, viz.
\begin{equation}	\label{eq:perpetuity}
X	\stackrel{d}{=} C + T_{1} X_{1},
\end{equation}
occurs when the number of positive $T_i$ is at most one, i.e.
\begin{equation}	\label{eq:N}
N := \sum_{i \geq 1}\1_{\{T_{i}>0\}} \le 1\quad\text{almost surely}
\end{equation}
and $\Prob(T_{1}=1,C=0)<1$. Then a solution, if it exists, is called perpetuity. Perpetuities have been studied in a series of papers, see \cite{GM2000} for a comprehensive overview of the existing literature. Theorem 3.1 in \cite{GM2000} provides a complete description of $\Fsum(C,T)$ in the case $\Prob(N \leq 1) = 1$. Therefore, we will exclude this case in what follows and assume that there is a positive probability for effective branching: $\Prob(N > 1) > 0$.

The next result covers some degenerate cases that will also be excluded from the analysis thereafter. Here and in what follows, we denote the Dirac measure at $c \in \R$ by $\delta_c$.
 
\begin{proposition}	\label{Prop:EN_leq_1}
Suppose that $\Prob(N > 1) > 0$.
\begin{itemize}
	\item[(a)]
		If $\E N \leq 1$, then $W^* < \infty$ almost surely
		and (the distribution of) $W^*$ is the unique solution to \eqref{eq:SumFP}.
	\item[(b)]
		If $\E N > 1$ but $\Prob(T_i \in \{0,1\} \text{ for all } i \geq 1) = 1$, then
		\begin{equation*}
		\Fsum(C,T)	=	
			\begin{cases}
			\emptyset	&	\text{if } \Prob(C > 0) > 0,	\\
			\{\delta_0\}	&	\text{if } \Prob(C = 0) = 1.
			\end{cases}
		\end{equation*}
\end{itemize}
\end{proposition}
\begin{proof}
(a) If $\E N \leq 1$, $\Fsum(0,T) = \{\delta_0\}$ by Lemma 3.1 in \cite{Liu1998}. On the other hand, $\E N \leq 1$ implies that $N_n := \sum_{|v|=n} \1_{\{L(v)>0\}}$ is a non-trivial critical or subcritical Galton-Watson process. Thus $\{v \in \V: L(v) > 0\}$ is almost surely finite entailing $W^* < \infty$ almost surely regardless of the choice of $C$. This shows (a) in view of the one-to-one correspondence between $\Fsum(0,T)$ and $\Fsum(C,T)$ provided by Theorem \ref{Thm:M=exp(-W*)xM_{hom}}.

\noindent
(b) If $\E N > 1$ but $\Prob(T_i \in \{0,1\} \text{ for all } i \geq 1) = 1$,
then Lemma 1.1 in \cite{Liu1998} shows that $\Fsum(0,T) = \{\delta_0\}$ if $C=0$ almost surely. Now assume that $C > 0$ with positive probability. We claim that $W^* = \infty$ almost surely on $S$, the set of survival. To prove this claim, using truncation if necessary, we can restrict ourselves to the case that $0 < c := \E C < \infty$. Further, we can assume without loss of generality that $N$ is bounded since otherwise one can replace $T$ by $\tilde{T} := (T_1,\ldots,T_m,0,0,\ldots)$ for some sufficiently large $m \in \N$. If one shows that $W_m^* = \infty$ almost surely on $S_m$, where $W^*_m$ is defined as in \eqref{eq:W*} but for the weighted branching process based on the truncated basic sequence $(C,T_1,\ldots,T_m,0,0,\ldots)$ instead of the original sequence $(C,T)$ and $S_m$ is the set of survival of the truncated process, then $W^* \geq W^*_m = \infty$ almost surely on $\bigcup_{m \geq 1} S_m$ which differs from $S$ by a $\Prob$-null set only.
Now assuming that $N$ is bounded above by some (large) constant, we infer that $Z_n := \sum_{|v|=n} L(v)$, $n \geq 0$ is a Galton-Watson process. With $m := \E N > 1$, we infer from the Kesten-Stigum theorem that $Z_n / m^n \to Z_\infty$ almost surely for some random variable $Z_\infty$ satisfying $Z_\infty > 0$ almost surely on $S$. Therefore,
\begin{equation*}
\liminf_{n \to \infty} \frac{Z_{n+1}}{Z_1+\ldots+Z_n} > 0
\quad	\text{almost surely on } S.
\end{equation*}
Now let $R_n := Z_n^{-1} \sum_{|v|=n}L(v)[C(v)-c]$, $n \geq 0$ on $S$ and $0$ otherwise. Then Proposition 4.1 in \cite{Ner1981} implies that
\begin{equation*}
\sum_{n \geq 1} \Prob(|R_n|>\delta | \A_n) < \infty
\quad	\text{almost surely on } S
\end{equation*}
for any $\delta > 0$.
Using Proposition 4.3 in \cite{Ner1981} (a variant of L\'evy's generalised Borel-Cantelli lemma), we infer that $\Prob(|R_n| > \delta \text{ infinitely often}|S) = 0$. Since $\delta > 0$ was arbitrary, we infer that $R_n \to 0$ almost surely on $S$. Consequently,
\begin{eqnarray*}
W^*
& = &
\sum_{n \geq 0} \sum_{|v|=n} L(v) C(v)	\\
& = &
\sum_{n \geq 0} Z_n\left(c + R_n \right)
= \infty
\quad	\text{almost surely on } S.
\end{eqnarray*}
\end{proof}

\section{Necessary conditions for $W^* < \infty$}	\label{sec:nec_conditions}

In view of the discussion in the last section it is natural to assume throughout that
\begin{equation}	\tag{A1}	\label{eq:A1}
\Prob	\big(T \in \{0,1\}^{\N} \big) < 1
\end{equation}
and
\begin{equation}	\tag{A2}	\label{eq:A2}
\E N > 1.
\end{equation}
However, further notation and conditions on the sequence $T$ are needed
to formulate our results. We assume without loss of generality that $N$ satisfies
\begin{equation}\label{eq:T_i>0<=>ileqN}
N = \sum_{i \geq 1} \1_{\{T_i > 0\}}	=	\sup\{i \geq 1: T_i > 0\}.
\end{equation}
We then define the function
\begin{equation}	\label{eq:m}
m:[0,\infty) \to [0,\infty],
\quad	\theta \mapsto	\E \sum_{i=1}^N T_i^{\theta}.
\end{equation}
A condition which is intimately connected with the existence of non-trivial fixed points of the homogeneous equation \eqref{eq:SumFP_hom} is that
\begin{equation}	\tag{A3a}	\label{eq:A3}
\inf_{\theta \in [0,1]} m(\theta) \leq 1.
\end{equation}
A stronger condition that will be used later is that the sequence $T$ possesses a \emph{characteristic exponent} $\alpha\in (0,1]$ which means that
\begin{equation}	\tag{A3b}	\label{eq:A3+}
\text{there is an } \alpha \in (0,1] \text{ such that } m(\beta)>m(\alpha)=1\text{ for }\beta\in [0,\alpha).
\end{equation}
Our next result shows that \eqref{eq:A3} is necessary for both the existence of a non-trivial fixed point of the homogeneous smoothing transform (part (a)) and the existence of a fixed point of the inhomogeneous smoothing transform (part (b)).

\begin{theorem}	\label{Thm:Liu_better}
Assume that \eqref{eq:A1} and \eqref{eq:A2} hold. Then \eqref{eq:A3} is necessary for
\begin{itemize}
	\item[(a)]
		$\Fsum(0,T) \backslash\{\delta_0\}\ne\emptyset$, and it is
		even equivalent if $N < \infty$ almost surely.
	\item[(b)]
		$\Fsum(C,T) \not = \emptyset$ and thus for $W^* < \infty$ almost surely,
		provided that $C>0$ with positive probability.
\end{itemize}
\end{theorem}

Liu \cite[Theorem 1.1]{Liu1998} proved that, if $N<\infty$ almost surely, then \eqref{eq:A3} implies the existence of a non-trivial fixed point of \eqref{eq:SumFP_hom}, that is, $\Fsum(0,T) \backslash \{\delta_0\}\ne\emptyset$. He further proved that, if $\E N < \infty$ and $m'(1) := \E \sum_{i=1}^N T_i \log^+ T_i<\infty$, then \eqref{eq:A3} is also necessary for $\Fsum(0,T) \backslash \{\delta_0\}\ne\emptyset$. Theorem \ref{Thm:Liu_better} generalises the last assertion of Liu's theorem by showing necessity of \eqref{eq:A3} under \eqref{eq:A1} and \eqref{eq:A2} only.

It is further worth mentioning that, if $N=\infty$ with positive probability, then $\sup_{\theta \in [0,1]} m(\theta) \leq 1$ is no longer sufficient for $\Fsum(0,T) \backslash \{\delta_0\}\ne\emptyset$. A counterexample is presented at the end of this section.

The proof of Theorem \ref{Thm:Liu_better} is furnished by two subsequent lemmata. For $n \in \N_0$, we define $W_n^{(\theta)} := \sum_{|v|=n} L(v)^{\theta}$ and then $W^{(\theta)} := \liminf_{n \to \infty} W_n^{(\theta)}$ ($\theta \geq 0$). The first lemma provides an estimate in terms of the random variable $W^{(1)}$ for Laplace transforms satisfying the inequality
\begin{equation}	\label{eq:subharmonic}
\varphi(t)	\leq	\E \prod_{|v|=n} \varphi(L(v)t)	\quad	\text{for all } t \geq 0.
\end{equation}

\begin{lemma}	\label{Lem:subharmonic}
Suppose that, as $n\to\infty$, $L_n^* := \sup_{|v|=n} L(v) \to L^*_{\infty} < \infty$ almost surely. Let $\varphi$ be the Laplace transform of a probability distribution on $[0,\infty)$. Put $D_1(t) := t^{-1}(1-\varphi(t))$ for $t>0$ and $D_1(0) := \lim_{t \downarrow 0} D_1(t) \in [0,\infty]$. If $\varphi$ satisfies \eqref{eq:subharmonic}, then
\begin{equation}	\label{eq:crucial_ineq}
\varphi(t)	\leq	\E \exp\Big(-t D_1(L_{\infty}^*t) W^{(1)}\Big)
\end{equation}
for all $t > 0$, where $0 \cdot \infty := 0$.
\end{lemma}
\begin{proof}
For this proof, it is stipulated that $\sum_{|v|=n}$ denotes summation over all $|v|=n$ with $L(v) > 0$. Note that, by the convexity of $\varphi$, $D_1$ is decreasing so that $D_1(0)$ actually exists in $[0,\infty]$. By Eq.\ \eqref{eq:subharmonic} and the inequality $\log(1+x) \leq x$ for $x > -1$, we infer that
\begin{align*}
\varphi(t)
&\leq 
\E \exp \Bigg(\sum_{|v|=n} \log(1+\varphi(L(v)t) - 1) \Bigg)	\\
&\leq
\E \exp \Bigg(-\sum_{|v|=n} (1-\varphi(L(v)t)) \Bigg)	\\
& =
\E \exp \Bigg(-t \sum_{|v|=n} L(v) \, \frac{1-\varphi(L(v)t)}{L(v)t} \Bigg)	\\
&\leq
\E \exp \Bigg(-t D_1(L_n^* t) \sum_{|v|=n} L(v) \Bigg).
\end{align*}
Since the integrand in the last line is bounded by $1$, we can use the $\limsup$-version of Fatou's lemma to obtain that
\begin{align*}
\varphi(t)
& \leq
\limsup_{n \to \infty} \E \exp \Bigg(-t D_1(L_n^* t) \sum_{|v|=n} L(v) \Bigg)	\\
& \leq
\E \exp \Bigg(-t \liminf_{n \to \infty} D_1(L_n^* t) \sum_{|v|=n} L(v) \Bigg)\\
&\leq	\E \exp \Bigg(-t D_1(L_{\infty}^* t) W^{(1)} \Bigg)
\end{align*}
with the convention that $0 \cdot \infty = 0$.
\end{proof}

\begin{lemma}	\label{Lem:subharmonic2}
Assume that \eqref{eq:A1}, \eqref{eq:A2} hold and further that $\varphi$ is the Laplace transform of a probability distribution on $[0,\infty)$ satisfying \eqref{eq:subharmonic}. Then for $\varphi$ to be non-constant it is necessary that \eqref{eq:A3} holds.
\end{lemma}
\begin{proof}
Assuming $\inf_{0 \leq \theta \leq 1} m(\theta) > 1$ we must show that $\varphi(t) = 1$ for all $t \geq 0$. To this end, we use a truncation argument. For $a \in \N$, define $\aT := (\aTi)_{i \geq 1}$ with
\begin{equation*}
\aTi :=	\begin{cases}
					\min\{a,T_i\}	&	\text{if } i \leq a,	\\
					0		&	\text{otherwise.}
					\end{cases}
\end{equation*}
The quantities $\am$, $\aN$, $\aW^{(\theta)}$, etc. are supposed to be defined as their counterparts without left subscript $a$ but w.r.t.\ $\aT$ instead of $T$. Then $\am$ is finite and convex on $[0,\infty)$ and, by the monotone convergence theorem, $\am(\theta) \uparrow m(\theta)$ as $a \to \infty$ for all $\theta \geq 0$. We claim that even
\begin{equation}	\label{eq:infam(theta)>1}
\inf_{0 \leq \theta \leq 1} \am(\theta) > 1
\end{equation}
holds true for all sufficiently large $a$. \\
To prove this claim suppose the contrary, \textit{i.e.}\ $\inf_{0 \leq \theta \leq 1} \am(\theta)\leq 1$ for all $a \in \N$. It follows that, for each $a \in \N$, there exists $\theta_a \in [0,1]$ such that $\am(\theta_a) \leq 1$. Let $\theta^*$ denote the limit of a convergent subsequence of the $\theta_a$. Then $\am(\theta^*)\uparrow m(\theta^{*})>1$ ensures $\am(\theta^*) > 1$ for some sufficiently large $a$ and therefore, by continuity, $\am(\theta) > 1$ for all $\theta$ in $(\theta^*-\varepsilon,\theta^*+\varepsilon)$ for some $\varepsilon > 0$. On the other hand, 
$|\theta_b - \theta^*| < \varepsilon$ must hold for some $b>a$ by the definition of $\theta^{*}$ which then leads to the impossible conclusion that $1 \geq\,\! _{b}m(\theta_b) \geq \am(\theta_b) > 1$. Therefore our claim must be true, that is, \eqref{eq:infam(theta)>1} holds for all sufficiently large $a$. In what follows, we fix any such $a>0$ and proceed by  distinguishing three cases:

\vspace{0.1cm}
\noindent
\textsc{Case 1}.	$\am(\theta) > 1$ for all $\theta > 0$ and $\Prob(T_i > 1) > 0$ for some $i \in \N$.	\\
In this case we have $m(\theta) > 1$ for all $\theta \geq 0$ and $\Prob(S(i) < 0) > 0$ for some $i \in \N$, where $S(v)=-\log L(v)$ should be recalled. From Corollary 1 in \cite{Big1998}, we then infer that $0>\lim_{n \to \infty} n^{-1}\inf_{|v|=n} S(v)=-\lim_{n \to \infty} n^{-1}\log L_{n}^{*}$ and thus $\lim_{n \to \infty} L_n^* = \infty$ almost surely on $S$, the set of survival of the BRW with positions $S(v)$. Note that $1-q := \Prob(S) > 0$, for $\E N = m(0) > 1$. Now for any $t > 0$, the dominated convergence theorem yields
\begin{equation*}
\varphi(t)	\leq	\liminf_{n \to \infty} \E \prod_{|v|=n} \varphi(L(v)t)
\leq	\lim_{n \to \infty} \E \varphi(L_n^*t)	=	q + (1-q)\varphi(\infty).
\end{equation*}
But this can only be true if $\varphi(0) = \varphi(\infty) = 1$, for $\varphi$ is continuous on $[0,\infty)$. Hence the monotonicity of $\varphi$ implies $\varphi(t) = 1$ for all $t\geq0$.

\vspace{0.1cm}
\noindent
\textsc{Case 2}.	$\am(\kappa) \leq 1$ for some $\kappa > 1$.	\\
Let $\kappa$ be the minimal positive number with $\am(\kappa) \leq 1$. Then $\kappa > 1$ and $\am$ is a strictly decreasing function on $[0,\kappa]$, giving $\am(1) > 1$ and $\am'(1) < 0$. Since $\aW^{(1)}_1 \leq a^2$, an appeal to Biggins' martingale convergence theorem (see \textit{e.g.}\ \cite{Lyo1997}) yields that
\begin{equation*}
\frac{\aW_n^{(1)}}{\am(1)^n} \to \aW^{(1)}	\quad	\text{almost surely,}
\end{equation*}
where $\aW^{(1)}$ is almost surely positive on $_aS$, the set of survival of the weighted branching process with weights $\aL(v)$, $v \in \V$. Consequently, $\aW^{(1)} = \lim_{n \to \infty} \aW^{(1)}_n = \infty$ almost surely on $_aS$. Since $\E \aN = \am(0) > 1$, we have that $1-\,\!_aq := \Prob(\,\!_aS) > 0$. Now observe that $\varphi$ satisfies inequality \eqref{eq:subharmonic} for the weight family $(\aL(v))_{v \in \V}$ and that $\aL_n^* \to 0$ almost surely by Theorem 3 in \cite{Big1998}. Hence the assumptions of Lemma \ref{Lem:subharmonic} are fulfilled and the lemma yields
\begin{equation*}
\varphi(t)	\leq	\E \exp(-t D_1(0) \aW^{(1)})	\quad	(t > 0)
\end{equation*}
with $D_1(0) = \lim_{s \downarrow 0} s^{-1}(1-\varphi(s))$. This inequality in combination with $\aW^{(1)} = \infty$ on $_aS$ implies that either $D_1(0) = 0$ or $\varphi(t) \leq \,\!_aq < 1$ for all $t > 0$. Since $\varphi$ is continuous, the first alternative must prevail and $\varphi(t) = 1$ for all $t \geq 0$.

\vspace{0.1cm}
\noindent
\textsc{Case 3}. $\am(\theta) > 1$ for all $\theta \geq 0$ and $\Prob(T_i \leq 1 \text{ for all } i \geq 1) = 1$.	\\
In this case, as in the previous one, $\am(1) > 1$, $\am'(1) < 0$, $\aW_1^{(1)} \leq a^2$ and, thus, $\aW^{(1)}=\infty$ on $\aS$. We want to argue as in the second case but we cannot conclude here from Theorem 3 in \cite{Big1998} that $\aL_n^* \to 0$ almost surely. However, $\aL_n^*$ is now decreasing in $n$ and therefore almost surely convergent to some random variable $\aL^*_{\infty}$ taking values in $[0,1]$. Again, the assumptions of Lemma \ref{Lem:subharmonic} are fulfilled and we conclude that
\begin{equation*}
\varphi(t)	\leq	\E \exp(-t D_1(\aL_{\infty}^*t) \aW^{(1)})
						\leq	\E \exp(-t D_1(t) \aW^{(1)})
\end{equation*}
for all $t>0$. Now suppose $\varphi(t_0) < 1$ for some $t_0>0$. Then $D_1(t) > 0$ and thus $\varphi(t) \leq \,\!_aq < 1$ for all $t > 0$, which contradicts the continuity of $\varphi$ at $0$. Hence $\varphi(t) = 1$ for all $t \geq 0$ must hold.
\end{proof}

Lemma \ref{Lem:subharmonic2} almost immediately implies Theorem \ref{Thm:Liu_better}:
\begin{proof}[Proof of Theorem \ref{Thm:Liu_better}]
Any $\varphi \in \Fsum(0,T)$ satisfies \eqref{eq:SumFE_hom} and, particularly, \eqref{eq:subharmonic}. Hence Lemma \ref{Lem:subharmonic2} implies that, if \eqref{eq:A3} fails, then $\varphi(t)=1$ for all $t \geq 0$ and, therefore, $\Fsum(0,T) = \{\delta_0\}$. By contraposition, \eqref{eq:A3} follows from $\Fsum(0,T) \backslash \{\delta_0\}\ne\emptyset$. This shows the first part of assertion (a). In order to conclude its second part we need to show that, provided $N < \infty$ almost surely, \eqref{eq:A3} is also sufficient for $\Fsum(0,T) \backslash \{\delta_0\}\ne\emptyset$. This is Theorem 1.1 in \cite{Liu1998}.

\noindent
To prove part (b) assume that $C>0$ with positive probability and note that any $\psi \in \Fsum(C,T)$ satisfies \eqref{eq:SumFE} and, therefore, also \eqref{eq:subharmonic}. By Lemma \ref{Lem:subharmonic}, if \eqref{eq:A3} fails, then $\psi(t)=1$ for all $t \geq 0$. Since $\psi(t) \leq \E e^{-tW^*}$ by Lemma \ref{Lem:Disintegration} and Corollary \ref{Cor:bounded_above}, we infer $W^*=0$ almost surely which contradicts the assumption that $C>0$ with positive probability. Thus, $\Fsum(C,T) \ne\emptyset$ implies \eqref{eq:A3}.
\end{proof}

Let us finally give an example which shows that assumption \eqref{eq:A3} is not sufficient for $\Fsum(0,T) \backslash \{\delta_0\}\ne\emptyset$ to hold when $N=\infty$ is allowed.

\begin{example}	\label{Exa:infm(theta)leq1_not_sufficient}
Let $T_i := c/((i+1) \log^2(i+1))$, $i \geq 1$ for some $c > 0$. Then $m(\theta) < \infty$ iff $\theta \geq 1$. Now we can choose $c>0$ so small that $m(1) < 1$. Thus, \eqref{eq:A1}, \eqref{eq:A2} and \eqref{eq:A3} are fulfilled. But $\Fsum(0,T) = \{\delta_0\}$ by Proposition 5.1 in \cite{AR2005}.
\end{example}

\section{Sufficient conditions for $\Fsum(C,T) \not = \emptyset$}	\label{sec:suf_conditions}

As in the previous section we assume validity of \eqref{eq:A1} and \eqref{eq:A2}. The following theorem provides sufficient conditions for $\Fsum(C,T) \not = \emptyset$. We focus on the case that $C > 0$ with positive probability since there is always the trivial solution $\delta_0$ otherwise and the case of non-trivial solutions has been investigated by Liu \cite[Theorem 1.1]{Liu1998}.

\begin{theorem}	\label{Thm:suf_conditions}
Suppose that \eqref{eq:A1} and \eqref{eq:A2} hold true. Then the following conditions are sufficient for $W^* < \infty$ and, therefore, for $\Fsum(C,T) \not = \emptyset$.
\begin{itemize}
	\item[(i)]
		$m(\beta) < 1$ and $\E C^{\beta} < \infty$ for some $0 < \beta \leq 1$.
	\item[(ii)]
		$m(\alpha) = 1$ for some $\alpha \in (0,1/5)$, $\E C < \infty$,
		$1 < m(\theta) < \infty$ for all $\alpha\ne\theta \in [-\varepsilon, \alpha+\varepsilon]$
		and some $\varepsilon > 0$, and $\E N^{1+\delta} < \infty$ for some $\delta > 0$.
\end{itemize}
\end{theorem}
Condition (i) is Lemma 4.4(i) in \cite{JO2010b}. For the reader's convenience, we provide a proof here.
\begin{proof}
Assume first that (i) holds. Then, using the subadditivity of $x \mapsto x^{\beta}$ ($x \geq 0$), we infer
\begin{eqnarray*}
\E (W^*)^{\beta}
& = &
\E \left[ \sum_{n \geq 0} \sum_{|v|=n} L(v) C(v) \right]^{\beta}	\\
& \leq &
\E \sum_{n \geq 0} \sum_{|v|=n} L(v)^{\beta} C(v)^{\beta}	
= \E C^{\beta} \, \sum_{n \geq 0} m(\beta)^n
< \infty.
\end{eqnarray*}
Now assume that (ii) holds. Then Theorem 1.2 in \cite{HS2009} guarantees that
\begin{eqnarray}
\frac{1}{2 \alpha}
& = &
\liminf_{n \to \infty} \frac{1}{\log n} \inf_{|v|=n} S(v)	\label{eq:HSliminf}	\\
& \leq &
\limsup_{n \to \infty} \frac{1}{\log n} \inf_{|v|=n} S(v)
= \frac{3}{2 \alpha}
\quad	\text{almost surely on } S.	\label{eq:HSlimsup}
\end{eqnarray}
In particular, for any $c>1$
\begin{equation*}
L_n^*	= \sup_{|v|=n} L(v) =  \exp\bigg(-\inf_{|v|=n} S(v)\bigg)
\leq e^{-\log n/(2 c \alpha)}
= n^{-1/(2c\alpha)}
\end{equation*}
for all sufficiently large $n$ almost surely on $S$. Moreover, by Markov's inequality,
\begin{equation*}
\Prob\Bigg(\sum_{|v|=n} L(v)^{\alpha} C(v) > n^c \Bigg)
\leq n^{-c} \E \sum_{|v|=n} L(v)^{\alpha} C(v)
= n^{-c} \E C
\end{equation*}
for all $n \geq 0$. Thus, by an application of the Borel-Cantelli Lemma,
\begin{equation*}
\Prob\Bigg(\sum_{|v|=n} L(v)^{\alpha} C(v) > n^c \text{ infinitely often} \Bigg)
= 0.
\end{equation*}
Further,
\begin{eqnarray*}
W^*
& = &
\sum_{n \geq 0} \sum_{|v|=n} L(v) C(v)
\leq \sum_{n \geq 0} (L_n^*)^{1-\alpha} \sum_{|v|=n} L(v)^{\alpha} C(v).
\end{eqnarray*}
A combination of the previously collected facts finally yields
\begin{equation*}
(L_n^*)^{1-\alpha} \sum_{|v|=n} L(v)^{\alpha} C(v)
\leq n^{1/(2c)-1/(2c\alpha)+c}
\end{equation*}
for sufficiently large $n$ almost surely. Hence, $W^*$ is almost surely finite if $\alpha < 1/(1+2c(1+c))$ which can be arranged by choosing $c$ sufficiently close to $1$, for $\alpha < 1/5$ by assumption.
\end{proof}

\section{Description of $\Fsum(C,T)$}	\label{sec:description}

In view of Theorem \ref{Thm:Liu_better} it is natural to determine $\Fsum(C,T)$ under
\eqref{eq:A1}, \eqref{eq:A2} and \eqref{eq:A3}. On the other hand, the last condition does not guarantee a universal answer as seen by Example \ref{Exa:infm(theta)leq1_not_sufficient}. Therefore, we replace \eqref{eq:A3} by the stronger counterpart \eqref{eq:A3+} hereafter. In fact, under a weak additional assumption, we can then determine $\Fsum(C,T)$. This assumption is that
\begin{equation}	\tag{A4}	\label{eq:A4}
\eqref{eq:A4a} \text{ or } \eqref{eq:A4b} \text{ holds}
\end{equation}
where \eqref{eq:A4a} and \eqref{eq:A4b} are as follows:
\begin{equation}	\tag{A4a}	\label{eq:A4a}
\E \sum_{i \geq 1} T_i^{\alpha} \log T_i \in (-\infty,0)
\text{ and }
\E \Big(\sum_{i \geq 1} T_i^{\alpha}\Big) \log^+ \Big(\sum_{i \geq 1} T_i^{\alpha}\Big) < \infty;
\end{equation}
and
\begin{equation}	\tag{A4b}	\label{eq:A4b}
N < \infty	\text{ almost surely	and	there is some } \theta<\alpha \text{ satisfying }
m(\theta) < \infty.
\end{equation}
Before we proceed with our theorem and thus provide a complete answer to Question (2) under \eqref{eq:A1}, \eqref{eq:A2}, \eqref{eq:A3+} and \eqref{eq:A4}, we need to introduce some notation necessitated by the fact that the closed multiplicative subgroup of $\R_{>0}$ generated by the positive $T_i$ may either be of the form $s^{\Z}$ for some $s > 1$ or be $\R_{>0}$ itself. We define $\G(T)$ as the intersection of all closed multiplicative groups $G \subseteq \R_{>0}$ such that $\Prob(T_i \in G \cup \{0\} \text{ for all } i \geq 1) = 1$. Let $r:=s>1$ if $\G(T) = s^{\Z}$ and $r:=1$ if $\G(T) = \R_{>0}$. Further, let $\mathfrak{P}_r$ denote the set of positive constant functions on $\R_{>0}$ if $r=1$ and let $\mathfrak{P}_r$ be the set of positive, multiplicatively $r$-periodic functions with a completely monotone derivative if $r>1$.

\begin{theorem}	\label{Thm:characterization_of_Fsum}
If \eqref{eq:A1}, \eqref{eq:A2}, \eqref{eq:A3+} and \eqref{eq:A4} hold and $W^*<\infty$ almost surely, then there exists a $\bT$-measurable random variable $W$ satisfying
\begin{equation}
W	=	\sum_{i \geq 1} T_i^{\alpha} [W]_i	\quad	\text{almost surely}
\end{equation}
such that the set of Laplace transforms $\psi$ of solutions to \eqref{eq:SumFP} is given by the family
\begin{equation*}
\psi(t)
=	\E \exp\Big(-tW^* - h(t)t^{\alpha} W\Big),	\quad 	t \geq 0,
\end{equation*}
parametrised by $h \in \mathfrak{P}_r$.
\end{theorem}
\begin{proof}
The theorem immediately follows from Lemma \ref{Lem:Disintegration} and Theorem \ref{Thm:M=exp(-W*)xM_{hom}} in combination with \cite[Theorem 8.3 and the proof of Corollary 2.3]{ABM2010}.
\end{proof}

The above result can be paraphrased as follows. Under mild conditions on $(C,T_1,T_2,\ldots)$ (including $\G(T)=\R_{>0}$ for convenience), there exists a non-negative random variable $W$ which is a function of $\bC \otimes \bT$ satisfying \eqref{eq:SumFP_hom} such that any solution $X$ to \eqref{eq:SumFP} has a representation of the form
\begin{equation}	\label{eq:representation_X}
X \stackrel{d}{=} W^* + hW^{1/\alpha} Y
\end{equation}
where $h \geq 0$ and $Y$ is independent of $\bC \otimes \bT$ with a one-sided stable law of index $\alpha$ (with Laplace transform $\exp(-t^{\alpha})$) if $\alpha \in (0,1)$, and $Y=1$ if $\alpha = 1$.

\section{Applications}	\label{sec:Applications}

In this section, we return to the examples from the introduction and discuss them in the context of the previously obtained results.

\begin{example}[Total population of a Galton-Watson process]	\label{Exa:total_population_(sub-)critical_GWP_2}
As in Example \ref{Exa:total_population_(sub-)critical_GWP}, let $(Z_n)_{n \geq 0}$ be a subcritical or non-trivial critical Galton-Wat\-son process with a single ancestor and total population size $X := \sum_{n \geq 0} Z_n$. Then $N = Z_1$ in the given situation.
If $\Prob(Z_1 \leq 1) = 1$, then we are in the situation of perpetuities and uniqueness of the solution to \eqref{eq:total_population_(sub-)critical_GWP} follows from Remark 2.4 in \cite{GM2000}, whereas in the case $\Prob(Z_1 > 1) > 0$ it follows from Proposition \ref{Prop:EN_leq_1}(a).
\end{example}

\begin{example}[Busy period in the M/G/1 queue]	\label{Exa:M/G/1_2}
In the situation of Example \ref{Exa:M/G/1_2}, the same arguments as in Example \ref{Exa:total_population_(sub-)critical_GWP_2} show that Eq.\ \eqref{eq:M/G/1} uniquely characterizes the distribution of the duration $X$ of the busy period in the M/G/1 queue. This result is not new, see \textit{e.g.}\ \cite[Example XIII.4(a)]{Fel1971}. The distribution of the busy period can even be explicitly calculated from \eqref{eq:M/G/1}, see \cite[pp.\;61--62]{Tak1962}.
\end{example}

\begin{example}[PageRank]	\label{Exa:PageRank2}
References \cite{VL2010} and \cite{JO2010b} are concerned with the PageRank tail behaviour. In our notation this means the tail behaviour of the random variable $W^*$ defined in \eqref{eq:W*}. In both of these papers the random variable $W^*$ is shown to be the unique solution to \eqref{eq:SumFP} subject to some additional (moment) constraints, see \cite[Theorem 3.1]{VL2010} and \cite[Lemma 4.5]{JO2010b}. Our Theorem \ref{Thm:characterization_of_Fsum} shows that in the situation of the cited results, there can be further solutions that do not satisfy the corresponding moment conditions and exhibit a different, in fact heavier tail behaviour. For instance, suppose that \eqref{eq:A1}, \eqref{eq:A2}, \eqref{eq:A3+}, and \eqref{eq:A4} are satisfied with characteristic exponent $\alpha < 1$. Suppose further that, for some $\beta > 1$, $m(\beta) = 1$, $m'(\beta) \in (0,\infty)$ (where $m'(\beta)$ denotes left derivative of $m$ at $\beta$ if $\beta$ is the right endpoint of the set $\{m < \infty\}$) and $\E C^\beta < \infty$. If, moreover, $\E(\sum_{i \geq 1} T_i)^{\beta} < \infty$, then Theorem 4.6(a) (in the case $\G(T) = \R_{>0}$) applies and yields that $\Prob(W^* > t) \sim Ht^{-\beta}$ as $t \to \infty$.
On the other hand, by Theorem \ref{Thm:characterization_of_Fsum}, there are further solutions with representations of the form $W^* + hW^{1/\alpha} Y$, $h \geq 0$ as in \eqref{eq:representation_X}. Denote the Laplace transform of $W$ by $\varphi$. Then $1-\varphi(t)$ is regularly varying of index $1$ at $0$, see \cite[Theorem 3.1]{ABM2010}. We claim that
\begin{equation}	\label{eq:tails}
\Prob(W^* + hW^{1/\alpha} Y > t) \sim \frac{h^{\alpha}}{\Gamma(1-\alpha)}(1-\varphi(t^{-\alpha}))	\quad	\text{as } t \to 0.
\end{equation}
We will briefly indicate how this can be proved. Since $hW^{1/\alpha}Y$ has heavier tails than $W^*$, it is easy to see that only the asymptotic behaviour of $hW^{1/\alpha} Y$ matters. Moreover, since $1-\varphi(t^{-\alpha})$ is regularly varying of index $-\alpha$ at $\infty$, it suffices to show that
$\Prob(W^{1/\alpha} Y > t) \sim \Gamma(1-\alpha)^{-1}(1-\varphi(t^{-\alpha}))$. To this end, it is convenient to first determine the asymptotics for $W$ and $Y$. Since $D(t) := (1-\varphi(t))/t$ is slowly varying at $0$, Theorem XIII.5.2 in \cite{Fel1971} together with Eq.\ (XIII.5.21) in the same reference provide us with
\begin{equation}
\int_0^t \Prob(W>u) \, {\mathrm d}u \sim D(t^{-1})	\quad	\text{as } t \to \infty.
\end{equation}
Using an argument as in the proof of the lemma on p.\ 446 in \cite{Fel1971} (notice that the lemma itself does not apply since $\rho = 0$ in the notation of \cite{Fel1971}), one can see that $\Prob(W>t) = o(1-\varphi(t^{-1}))$ as $t \to \infty$. 
Further, from \cite[Theorem XIII.6.1]{Fel1971}, we infer that $\Prob(Y > t) \sim \Gamma(1-\alpha)^{-1}t^{-\alpha}$ as $t \to \infty$. Now putting things together, for any given $\varepsilon > 0$ and some sufficiently large $K = K(\varepsilon)$, we infer
\begin{eqnarray*}
\Prob(W^{1/\alpha} Y > t)
& = &
\E [\1_{\{W > 0\}} \Prob(Y > t W^{-1/\alpha}|W)]	\\
& \leq &
\frac{1+\varepsilon}{\Gamma(1-\alpha)} (t/K)^{-\alpha} \E W \1_{\{W \leq (t/K)^{\alpha}\}} + \Prob(W > (t/K)^{\alpha})	\\
& \sim &
\frac{1+\varepsilon}{\Gamma(1-\alpha)}(1-\varphi(t^{-\alpha}))	\quad	\text{as } t \to \infty.
\end{eqnarray*}
Letting $\varepsilon \to 0$ gives the upper bound, and the lower bound can be derived similarly.
\end{example}

\section{Concluding Remarks}	\label{sec:remarks}

Let us finish with a brief discussion of the necessary and sufficient conditions for the almost sure finiteness of $W^*$ presented in Sections \ref{sec:nec_conditions} and \ref{sec:suf_conditions}. As a necessary condition we obtained $\inf_{0 \leq \theta \leq 1} m(\theta) \leq 1$ in Theorem \ref{Thm:Liu_better}(b), while, in reverse direction, we showed in Theorem \ref{Thm:suf_conditions} that, if the function $m$ drops strictly below $1$ in the interval $[0,1]$ and $C$ satisfies a mild moment condition, then $W^*<\infty$ almost surely. This raises the question of what happens in the case $\inf_{0 \leq \theta \leq 1} m(\theta) = 1$. We showed sufficiency of Condition (ii) in Theorem \ref{Thm:suf_conditions} in order to show that there are actually cases where $\inf_{0 \leq \theta \leq 1} m(\theta) = 1$ and still $W^* < \infty$ almost surely holds true. On the other hand, if \eqref{eq:A4} and \eqref{eq:A4a} hold with $\alpha=1$ and $C$  equals a positive constant almost surely, then
\begin{equation*}
W^*	=	C \sum_{n \geq 0} W_n^{(1)} = \infty
\quad	\text{almost surely on } S
\end{equation*}
owing to the fact that $W_n^{(1)} \to W^{(1)}$ almost surely and $W^{(1)}>0$ almost surely on $S$, see \textit{e.g.}\ \cite{Lyo1997}. This indicates that the situation is more involved when $\inf_{0 \leq \theta \leq 1} m(\theta) = 1$.

\section*{Acknowledgements}

The authors are grateful to Takis Konstantopoulos for pointing out a reference, and to Ralph Neininger for pointing out an error in an earlier version of this paper.

\bibliographystyle{abbrv}
\bibliography{Fixed_points}

\end{document}